\documentclass[11pt,letterpaper]{amsart}

\usepackage{graphicx, color}

\usepackage{amssymb}

\usepackage{amsmath}

\usepackage{latexsym}\usepackage{latexsym} \makeatletter

\makeatletter
\@namedef{subjclassname@2020}{%
  \textup{2020} Mathematics Subject Classification}
\makeatother

\newtheorem{thm}{Theorem}

\newtheorem{lem}{Lemma}

\newtheorem{pro}{Proposition}

\newtheorem{definition}[thm]{Definition}

\newtheorem{exa}{Example}

\title[Plaque topological stability]{On the plaque topological stability of partially hyperbolic diffeomorphisms}

\author{L. Li}
\address{Department of Mathematics, SUSTech, Shenzhen, 518000, China.}
\email{12131228@mail.sustech.edu.cn}

\author{C.A. Morales}
\address{Hangzhou International Innovation Institute of Beihang University,  
Hangzhou 311115, China}
\email{morales@impa.br}

\author[B. Shin]{B. Shin}
\address{Department of Mathematics, SungKyunKwan University, Suwon, 16419, Republic of Korea}
\email{bmshin@skku.edu}



\begin{document}

\maketitle

\begin{abstract}
We prove that every dynamically coherent plaque expansive partially hyperbolic diffeomorphism is topologically stable with respect to the central foliation (in short, {\em plaque topologically stable}).
Next, we study partially hyperbolic diffeomorphisms that are both expansive and topologically stable with respect to a central foliation.
We show that the center chain recurrent set for such diffeomorphisms
belongs to the closure of the center periodic points.
\end{abstract}


\section{Introduction}

\noindent
It is well-known that every uniformly hyperbolic diffeomorphism of a closed manifold is both structurally and topologically stable.
In contrast, partially hyperbolic diffeomorphisms are generally neither structurally nor topologically stable.
This has led some authors to search for alternative notions of stability.
In particular, \cite{hz} introduced the notion of topological quasi-stability and proved that it holds for every partially hyperbolic diffeomorphism.
See \cite{hzz} where another proof of this result using the notion of quasi-shadowing property is given. 

In this paper, we prove that every dynamically coherent, plaque expansive, partially hyperbolic diffeomorphism is topologically stable with respect to the central foliation (for short {\em plaque topologically stable}).
Furthermore, any partially hyperbolic diffeomorphism that is expansive and topologically stable with respect to a central foliation satisfies the following: We show that the center chain recurrent set for such diffeomorphisms
belongs to the closure of the center periodic points.
We now state our results precisely.

\subsection{Basics}
By {\em closed manifold} we mean a compact connected boundaryless $C^\infty$ Riemannian manifold $M$.
A diffeomorphism $f:M\to M$ is {\em partially hyperbolic} if there are constants
$\lambda<\hat{\gamma}<1<\gamma<\mu$, $C\geq1$, and
a $Df$-invariant tangent bundle splitting $TM=E^s\oplus E^c\oplus E^u$
such that for every $x\in M$,
\begin{eqnarray*}
&\|Df^n(x)v^s\|&\leq C\lambda^n\|v^s\|\quad\quad\forall v^s\in E^s_x,\, n>0;\\
C^{-1}\hat{\gamma}^n\|v^c\|\leq &\|Df^n(x)v^c\|&\leq C\gamma^n\|v^c\|\quad\quad\forall v^c\in E^c_x,\, n>0;\\
C^{-1}\mu^n\|v^u\|\leq&\|Df^n(x)v^u\|&\quad\quad\quad\quad\quad\quad\quad\forall v^u\in E^u_x,\, n>0.
\end{eqnarray*}

The stable $E^s$ and unstable $E^u$ subbundles are known to be H\"older continuous and uniquely integrable \cite{hps},
i.e. there exist unique ($C^0$) foliations $\mathfrak{F}^s$ and $\mathfrak{F}^u$, called (strong) stable and unstable, respectively, tangent to $E^s_x$ and $E^u_x$
for all $x\in M$. Furthermore, such foliations are invariant (i.e. their leaves are permuted by $f$).

Following \cite{bbi} and Definition 3 in \cite{kt}, we say that $f$ is {\em dynamically coherent} if the subbundles $E^s \oplus E^c$ and $E^c \oplus E^u$ are
uniquely integrable (some authors omit the uniqueness in this definition).
Partially hyperbolic diffeomorphisms are dynamically coherent in the $3$-torus \cite{bbi} but not in general closed $3$-manifolds(\footnote{Bonatti, C., Gogolev, A., Hammerlindl, A., Potrie, R., Anomalous partially hyperbolic diffeomorphisms III: Abundance and incoherence, {\em Geom. Topol.} 24 (2020), no. 4, 1751--1790.}).
Any invariant foliation tangent to the central direction $E^c$ will be referred to as a {\em center foliation}.
Center foliations need not exist \cite{wil}, and when they do, they are not necessarily unique.
However, if a partially hyperbolic diffeomorphism is dynamically coherent, then it admits a unique center foliation; see Section 2 of \cite{bw}.

\subsection{Set-valued maps and foliations}
We start with set-valued analysis \cite{af}.
Denote by $2^M$ the set of all subsets of $M$.
Maps $H:M\to 2^M$ with $H(x)\neq\emptyset$ for all $x\in M$ are called {\em set-valued}. If $H(x)$ has one point only for all $x\in X$, then we say that $H$ is {\em single-valued}. Clearly, maps $h:M\to M$ can be identified with single-valued maps $H:M\to 2^M$ via the induced set-valued map
$H(x)=\{h(x)\}$ for all $x\in M$. If $H(x)$ is compact for every $x\in M$, we say that $H$ is {\em compact-valued}.

Next, we consider a foliation $\mathfrak{F}$ of $M$.
Given a leaf $L$ of $\mathfrak{F}$
we denote by $d_L$ the intrinsic metric on $L$.
We say that $S\subset M$ is {\em saturated} if $\mathfrak{F}(x)\subset S$ for every $x\in S$. Denote by $\mathfrak{F}(x)$ the leaf of $\mathfrak{F}$ containing $x\in M$. Define
$$
\mathfrak{F}_\epsilon(x)=\{y\in \mathfrak{F}(x):d_{\mathfrak{F}(x)}(x,y)\leq\epsilon\},\quad\quad\forall \epsilon>0.
$$
This is the closed $\epsilon$-ball of $x$ in $\mathfrak{F}(x)$ with respect to the intrinsic metric.

We now define $\mathfrak{F}$-valued maps $H$. The natural way is to require that
$H(x)$ is contained in a leaf of $\mathfrak{F}$ for every $x\in M$ or, equivalently,
$$
H(x)\subset \bigcap_{y\in H(x)}\mathfrak{F}(y),\quad\quad\forall x\in M.
$$
However, we don't want $H(x)$ to be spread in the whole leaf $\mathfrak{F}(y)$ but rarther to be concentrated in the intrinsic ball $\mathfrak{F}_e(y)$ for some $e>0$. By doing so we obtain the following definition.

\begin{definition}
A set-valued map $H:M\to 2^M$ is {\em $\mathfrak{F}$-valued} if there exists
$e>0$ (called valuation constant) such that
$$
H(x)\subset \bigcap_{y\in H(x)}\mathfrak{F}_e(y),\quad\quad\forall x\in X.
$$
\end{definition}

\subsection{Continuity of set-valued maps with respect to foliations}
A natural way to define continuity for set-valued maps $H:M\to 2^M$ is to require that, for every $\rho>0$, there exists $\Delta>0$ such that if $x,x'\in M$ with $d(x,x')\leq \Delta$, then
$D(H(x),H(x'))\leq \rho$, where
\[
D(A,B)=\max\{\operatorname{dist}(A,B),\operatorname{dist}(B,A)\}
\]
is the Hausdorff distance between $A,B\subset M$. Here
\[
\operatorname{dist}(A,B)=\sup_{a\in A}\inf_{b\in B} d(a,b).
\]
If $A=\{a\}$ (resp. $B=\{b\}$) reduces to a single point, we write $\operatorname{dist}(a,B)$ (resp. $\operatorname{dist}(A,b)$) instead of $\operatorname{dist}(\{a\},B)$ (resp. $\operatorname{dist}(A,\{b\})$).

A stronger, though less common, way to define such a continuity is
to require that, for every $\rho>0$, there exists $\Delta>0$ such that if $x,x'\in M$ with $d(x,x')\leq \Delta$, then
\begin{equation}
\label{chucha}
d(y,y')\leq \rho \qquad \forall (y,y')\in H(x)\times H(x').
\end{equation}
Though both definitions reduce to the classical notion of continuity in the single-valued case, they are too restrictive for our purposes.

In the sequel, we modify the latter property to introduce a notion of continuity for set-valued maps $H:M\to 2^M$ with respect to a foliation $\mathfrak{F}$.
More precisely, if $H$ is $\mathfrak{F}$-valued, then the smaller the valuation constant $\epsilon$ is, the closer the local plaque $\mathfrak{F}_\epsilon(y')$ will be to the point $y'$.  
Thus, it is reasonable to replace $y'$ by $\mathfrak{F}_\epsilon(y')$ in \eqref{chucha} to obtain the following definition.

\begin{definition}
An $\mathfrak{F}$-valued map $H:M\to 2^M$ is said to be {\em continuous with respect to $\mathfrak{F}$} if the following property holds for some valuation constant $\epsilon$:  
for every $\rho>0$, there exists $\Delta>0$ such that whenever $x,x'\in M$ with $d(x,x')\leq \Delta$, we have
\[
\operatorname{dist}(y,\mathfrak{F}_\epsilon(y')) \leq \rho,
\qquad \forall (y,y')\in H(x)\times H(x'),
\]
where $\operatorname{dist}(y,\mathfrak{F}_\epsilon(y')):=\inf_{z\in \mathfrak{F}_\epsilon(y')} d(y,z)$.
We then say that $H$ is {\em continuous with respect to $\mathfrak{F}$ with valuation constant $\epsilon$}.
\end{definition}

This definition includes the classical continuity for single-valued maps:

\begin{exa}
A map $h:M\to M$ is continuous if and only if the induced set-valued map
$H:M\to 2^M$ defined by $H(x)=\{h(x)\}$ for all $x\in M$ is continuous with respect to the trivial foliation by points of $M$.
\end{exa}

\subsection{Topological stability with respect to foliations}
To motivate we recall the classical definition of topological stability \cite{w}.
Let $i_M:M\to M$ be the identity map.
The $C^0$ distance between maps $l,r:M\to M$ is defined by
$$
d_{C^0}(l,r)=\sup_{x\in M}d(l(x),r(x)).
$$
We say that a homeomorphism $f:M\to M$ is topologically stable if for every $\epsilon>0$ there is $\delta>0$ such that for every homeomorphism $g:M\to M$ with $d_{C^0}(f,g)<\delta$ there is a continuous map $h:M\to M$ (often called semiconjugation) such that
$$
d_{C^0}(h,id_M)\leq\epsilon
\quad\mbox{ and }\quad
f\circ h=h\circ g.
$$
In a potential definition of topological stability with respect to $\mathfrak{F}$
the role of semiconjugation will be played by $\mathfrak{F}$-valued maps $H:M\to 2^M$.
This forces us to consider the distance between such maps and the identity which we fix as
$$
d_{C^0}(H,id_M)=\sup_{x\in M}dist(H(x),x).
$$
For the conjugating rule $f\circ h=h\circ g$ the usual (e.g. \cite{lm}) is to consider the identity
$f\circ H=H\circ g$ which is equivalent to $f(H(x))=H(g(x))$ for every $x\in M$.
But due to the $\mathfrak{F}$-valued condition this identity implies the inclusion
$$
f(H(x))\subset \bigcup_{y\in H(g(x))}\mathfrak{F}_\epsilon(y),\quad\quad\forall x\in M,
$$
where $\epsilon$ is a valuation constant.
This inclusion is still good for us: the smaller the valuation constant $\epsilon$ is, the closer $H$ is to being a single-valued map, and the closer $f(H(x))$ and $H(g(x))$ will be.
With this in mind we state the following definition.

\begin{definition}
A homeomorphism $f:M\to M$
is {\em topologically stable with respect to $\mathfrak{F}$}
if for every $\epsilon>0$ there is $\delta>0$ with the following property:
for every homeomorphism $g:M\to M$ with $d_{C^0}(f,g)\leq\delta$,
there is a compact-valued map $H:M\to 2^M$ which is continuous with respect to $\mathfrak{F}$ with constant $\epsilon>0$ such that
\begin{equation}
\label{alto}
d_{C^0}(H,i_M)\leq\epsilon\quad\mbox{ and }\quad
f(H(x))\subset \bigcup_{y\in H(g(x))}\mathfrak{F}_{\epsilon}(y)
\quad(\forall x\in M).
\end{equation}
\end{definition}

At first glance we observe that this kind of stability includes the classical one as explained below.

\begin{exa}
\label{air}
Every topologically stable homeomorphism is topologically stable with respect to any foliation. The converse is also true since
a homeomorphism is topologically stable if and only if it is topologically stable with respect to the trivial foliation by points. 
\end{exa}

We use the above definition to introduce the following one.

\begin{definition}
A dynamically coherent partially hyperbolic diffeomorphism is {\em plaque topologically stable} if it is topologically stable with respect to its center foliation.
\end{definition}

\subsection{Results}
Our first result gives a sufficient condition for a dynamically coherent partially hyperbolic diffeomorphism to be plaque topologically stable.
Given such a diffeomorphism $f:M\to M$
and $\epsilon>0$, a $\epsilon$-center pseudo orbit is a bi-infinite sequence
$(x_i)_{i\in\mathbb{Z}}$ such that
$f(x_i)\in \mathfrak{F}^c_\epsilon(x_{i+1})$ for all $i\in\mathbb{Z}$.
We say that $f$ is {\em plaque expansive} \cite{hps} if there is $\epsilon>0$ such that if two $\epsilon$-center pseudo orbits $(x_i)_{i\in\mathbb{Z}}$ and $(y_i)_{i\in\mathbb{Z}}$ satisfy $d(x_i,y_i)<\epsilon$ for all $i\in\mathbb{Z}$, then $y_0\in \mathfrak{F}^c_\epsilon(x_0)$.

Every dynamically coherent plaque expansive partially hyperbolic diffeomorphism $f:M\to M$ is center leaf stable in the following sense: for every partially hyperbolic diffeomorphism $g:M\to M$ which is $C^1$ close to $f$ there is a homeomorphism $h:M\to M$ sending central leaves of $f$ onto central leaves of $g$ (c.f. \cite{hps} or \cite{psw}).
The following result gives another type of stability for such diffeomorphisms.

\begin{thm}
\label{taylor}
Every dynamically coherent plaque expansive partially hyperbolic diffeomorphism is plaque topologically stable.
\end{thm}

The well-known {\em plaque expansivity conjecture \cite{hps}} asserts that
every dynamically coherent partially hyperbolic diffeomorphism is plaque expansive.
If true, then we could remove the plaque expansivity hypothesis from this theorem.

In the uniformly hyperbolic case, this theorem reduces to the well-known topological stability of uniformly hyperbolic diffeomorphisms \cite{w1}.
In fact, every uniformly hyperbolic diffeomorphism is dynamically coherent, with the trivial foliation by points serving as the center foliation.
Since it is also expansive, the diffeomorphism is expansive with respect to this foliation.
Consequently, it is plaque expansive and therefore plaque topologically stable by the theorem.
Moreover, because topological stability with respect to the trivial foliation by points coincides with ordinary topological stability (Example \ref{air}), we conclude that the diffeomorphism is topologically stable, as claimed.

Another class of examples where this theorem applies is that of partially hyperbolic diffeomorphisms on the $3$-torus.
In fact, such diffeomorphisms are dynamically coherent \cite{bbi} and plaque expansive \cite{h}; hence, by the theorem, they are plaque topologically stable.

Our next result gives some consequences of the topological stability with respect to center foliations.
To motivate we recall some basic concepts in topological dynamics \cite{ah}. Let $f:M\to M$ be a homeomorphism.
We say that $x\in M$ is a {\em periodic point} if there is $n\in\mathbb{N}$ such that $f^n(x)=x$.
We say that $x$ is {\em nonwandering} if for every neighborhood $U$ of $x$ there is $n\in\mathbb{N}$ such that $f^n(U)\cap U\neq\emptyset$.
Given $\delta>0$, a {\em $\delta$-chain} from $x$ to $y$ is a finite sequence
$x_0,\cdots, x_r$ with $r\geq1$ such that $x_0=x$, $x_r=y$ and
$d(f(x_i),x_{i+1})\leq\delta$ for all $0\leq i\leq r-1$.
We write $x\approx_\delta y$ whenever
$\delta$-chains from $x$ to $y$ and viceversa exist.
We write $x\approx y$ if $x\approx_\delta y$ for every $\delta>0$.
We say that $x\in M$ is {\em chain recurrent} if
$x\approx x$. Denoting $Per(f)$, $\Omega(x)$ and $CR(f)$ the set of periodic, nonwandering and chain recurrent points respectively we have these sets are invariant, $\Omega(f)$ is closed nonempty and
$Per(f)\subset \Omega(f)\subset CR(f)$. See \cite{ah} for details.

Now, suppose that $f$ is a partially hyperbolic diffeomorphism with a prescribed center foliation $\mathcal{F}^c$.
Following p. 417 in \cite{hzz}
we say that $x\in M$ is {\em center nonwandering} if for every saturated open neighborhood $U$ of $x$ there is $n\geq1$
such that $f^n(U)\cap U\neq\emptyset$.
Moreover, $x$ is {\em center periodic}
if $\mathfrak{F}^c(f^k(x))=\mathfrak{F}^c(x)$ for some $k\geq1$.

Next, we introduce the following definitions:
Given $\delta>0$, an {\em $(\mathfrak{F}^c,\delta)$-chain} from $x$ to $y$ is a finite sequence
$x_0,\cdots, x_r$ with $r\geq1$ such that $x_0=x$, $x_r=y$ and
$$
dist(f(x_i),\mathfrak{F}^c_\delta(x_{i+1}))\leq\delta,\quad\quad\forall 0\leq i\leq r-1.
$$

On the other hand, write $x\approx_{\mathfrak{F}^c,\delta} y$ whenever
$(\mathfrak{F}^c,\delta)$-chains from $x$ to $y$ and vice versa exist.
We write $x\approx_{\mathfrak{F}^c} y$ if $x\approx_{\mathfrak{F}^c,\delta} y$ for every $\delta>0$.
We say that $x\in M$ is {\em center-chain recurrent} if
$x\approx_{\mathfrak{F}^c} x$.

Denoting $Per^c(f)$, $\Omega^c(x)$ and $CR^c(f)$ the set of center periodic, center nonwandering and center chain recurrent points respectively we have these sets are invariant, $\Omega^c(f)$ is closed nonempty and
$Per^c(f)\subset \Omega^c(f)$.
Furthermore,
$Per(f)\subset Per^c(f)$ and $\Omega(f)\subset \Omega^c(f)$.

We will prove the following result.

\begin{thm}
\label{cirus}
Let $f:M\to M$ be a partially hyperbolic diffeomorphism of a closed manifold. If $f$ is expansive and topologically stable with respect to a center foliation $\mathfrak{F}^c$, then
$CR^c(f)\subset\overline{Per^c(f)}\subset \Omega^c(f).$
\end{thm}

Since we are not assuming dynamical coherence, we cannot apply Theorem \ref{taylor} to remove the topological stability condition from this statement. In a previous version of this paper we claimed that $\Omega^c(f)\subset CR^c(f)$, which would allow the inclusions in the above statements to be replaced by equalities. However, we do not know whether the inclusion $\Omega^c(f)\subset CR^c(f)$ actually holds. We are grateful to the anonymous referee for bringing this to our attention.

We can also compare Theorem \ref{cirus} with Theorem D in \cite{hzz} where partially hyperbolic diffeomorphism with uniformly compact $C^1$ center foliation were considered. See also the recent $C^r$ chain closing lemma for partially hyperbolic diffeomorphisms \cite{sw}. Theorem \ref{cirus} can be applied to any partially hyperbolic diffeomorphisms of the $3$-torus (likewise those in Proposition 5 of \cite{b}).

\section{Proof of the theorems}

\noindent
We derive the theorems from three propositions about foliation-preserving homeomorphisms.
Let $M$ be a closed manifold.
Let $f:M\to M$ denote a homeomorphism and $\mathfrak{F}$ be a foliation of $M$.
We say that $f$ {\em preserves $\mathfrak{F}$} if $f$ permutes the leaves of $\mathfrak{F}$ i.e.
$f(\mathfrak{F}(x))=\mathfrak{F}(f(x))$ for every $x\in M$.
{\em In what follows we assume that $f$ preserves $\mathfrak{F}$}.

An $(\mathfrak{F},\epsilon)$-orbit of $f$ is a bi-infinite sequence
$(x_k)_{k\in\mathbb{Z}}$ satisfying
$$
f(x_k)\in \mathfrak{F}_\epsilon(x_{k+1}),\quad\quad\forall k\in\mathbb{Z}.
$$
The next definition can be found in \cite{hps}.

\begin{definition}
We say that
$f$ is {\em expansive with respect to $\mathfrak{F}$}
if for every $\epsilon_0>0$ there is $e>0$
(which we call {\em expansivity constant} for the given $\epsilon_0$) such that if $(x_k)_{k\in\mathbb{Z}}$ and $(y_k)_{k\in\mathbb{Z}}$ are $(\mathfrak{F},e)$-orbits
with
$$
d(x_k,y_k)\leq e,\quad\quad\forall k\in\mathbb{Z},
$$
then $y_0\in \mathfrak{F}_{\epsilon_0}(x_0)$.
\end{definition}

Actually \cite{hps} used the name "plaque expansivity" but here we deserve that name when considering the center foliation of a dynamically coherent partially hyperbolic diffeomorphism. We also choose the above name to emphasize the dependence on $\mathfrak{F}$.

This definition generalizes the notion of expansivity in topological dynamics.
Recall that a homeomorphism $f:M\to M$ is {\em expansive} \cite{ah} if there is $\epsilon>0$ such that $x=y$ whenever $x,y\in M$ and $d(f^n(x),f^n(y))<\epsilon$ for every $n\in\mathbb{Z}$.

\begin{exa}
\label{lotus}
A homeomorphism of a closed manifold is expansive if and only if it is expansive with respect to the foliation by points.
Every homeomorphism is expansive with respect to the foliation with only one leaf.
\end{exa}

\begin{proof}
The first assertion is easier to prove so we prove the second one only.
Let $M$ be a close manifold and $\mathfrak{F}=\{M\}$ be the foliation with the whole manifold as unique leaf.
Let $f:M\to M$ be a homeomorphism. Clearly preserves $\mathfrak{F}$.
Notice that $\mathfrak{F}_\epsilon(x)=B[x,\epsilon]$ is the closed $\epsilon$-ball centered at $x$
($\forall x\in M$).
Fix $\epsilon_0>0$ and take $e=\epsilon_0$.
Suppose that $(x_k)_{k\in\mathbb{Z}}$ and $(y_k)_{k\in\mathbb{Z}}$ are $(\mathfrak{F},e)$-orbits with
$d(x_k,y_x)\leq e$ for all $k\in \mathbb{Z}$.
In particular, $y_0\in B[x_0,\epsilon_0]=\mathfrak{F}_{\epsilon_0}(x_0)$ hence $f$ is expansive with respect to $\mathfrak{F}$.
\end{proof}

We will use the following result corresponding to Lemma 2 in \cite{w}.

\begin{lem}
\label{l1}
If a homeomorphism $f:M\to M$ is expansive with respect to $\mathfrak{F}$, then $f$ is {\em uniformly expansive with respect to $\mathfrak{F}$} in the following sense: For every $\epsilon_0>0$
there is $e>0$ (called {\em uniform expansivity constant} for the given $\epsilon_0$) such that for every $\rho>0$ there is $N\in\mathbb{N}$ such that if $(x_k)_{k=-N}^N$ and $(y_k)_{k=-N}^N$ are finite sequences satisfying
\begin{enumerate}
\item[(a)]
$f(x_k)\in \mathfrak{F}_e(x_{k+1})$ and $f(y_k)\in \mathfrak{F}_e(y_{k+1})$ for
$-N\leq k\leq N-1$;
\item[(b)]
$d(x_k,y_k)\leq e$ for $-N\leq k\leq N$,
\end{enumerate}
then
$dist(y_0,\mathfrak{F}_{\epsilon_0}(x_0))\leq \rho.$
\end{lem}

\begin{proof}
Fix $\epsilon_0>0$ and let
$e$ be an expansivity constant for this $\epsilon_0$.
If $e$ were not an uniform expansivity constant for $\epsilon_0$, there would exist $\rho>0$ such that for any $N\in \mathbb{N}$ there are finite sequences
$(x_k^N)_{k=-N}^N$ and $(y_k)_{k=-N}^N$ for $N\in\mathbb{N}$ satisfying
$f(z_k^N)\in \mathfrak{F}_e(z^N_{k+1})$ and $d(x_k^N,y_k^N)\leq e$
for $-N\leq k\leq N-1,\, z=x,y,$
but
\begin{equation}
\label{kurt}
dist(y_0^N,\mathfrak{F}_{\epsilon_0}(x_0^N))>\rho,\quad\forall N\in \mathbb{N}.
\end{equation}
Since $M$ is compact, by the Cantor diagonal argument, we can assume
that there are sequences $(x_k)_{k\in\mathbb{Z}}$ and $(y_k)_{k\in\mathbb{Z}}$
such that
$$
\lim_{N\to\infty}x^N_k=x_k\quad\mbox{ and }\quad  \lim_{N\to\infty}y^N_k=y_k.
$$
Fixing $k$ and letting $N\to\infty$ in (a) we get
$$
f(z_k)\in \mathfrak{F}_e(z_{k+1})\quad\quad(\forall k\in\mathbb{Z}).
$$
Then, $(z_k)_{k\in\mathbb{Z}}$ is a $(\mathfrak{F},e)$-orbit for $z=x,y$.
Likewise, fixing $k$ and letting $N\to\infty$ in (b) we get
$$
d(x_k,y_k)\leq e,\quad\quad\forall k\in\mathbb{Z}.
$$
Since  $e$ is an expansivity constant for $\epsilon_0$,
\begin{equation}
\label{kurdo}
y_0\in \mathfrak{F}_{\epsilon_0}(x_0).
\end{equation}
But by letting $N\to\infty$ in \eqref{kurt} we get
$$
dist(y_0,\mathfrak{F}_{\epsilon_0}(x_0))\geq \rho>0
$$
contradicting \eqref{kurdo}. This completes the proof.
\end{proof}

The next definition seems to be new.
It is inspired by the plaque expansivity and the central shadowing for partially hyperbolic diffeomorphisms \cite[Definition 8]{kt}.

\begin{definition}
\label{parate}
We say that $f$ has the {\em shadowing property with respect to $\mathfrak{F}$}
if for every $\epsilon>0$ there is $\delta>0$ such that
every $\delta$-pseudo orbit $(y_k)_{k\in\mathbb{Z}}$ can be {\em $\epsilon$-shadowed by an $(\mathfrak{F},\epsilon)$-orbit} $(x_k)_{k\in\mathbb{Z}}$ i.e.
$$
d(x_k,y_k)\leq \epsilon,\quad\quad\forall k\in\mathbb{Z}.
$$
\end{definition}

We consider the case of the foliation with only one leaf.

\begin{exa}
\label{interferiu}
Let $\mathfrak{F}=\{M\}$ be the foliation with the whole manifold as unique leaf.
Then, every homeomorphism $f:M\to M$ has the shadowing property with respect to $\mathfrak{F}$.
\end{exa}

\begin{proof}
Given $\epsilon>0$ we take $\delta=\epsilon$. Let $(y_k)_{\in\mathbb{Z}}$ be a $\delta$-pseudo orbit of $f$, i.e.,
$$
f(y_k)\in B[y_{k+1},\delta],\quad\quad\forall k\in\mathbb{Z}.
$$
Define $x_k=y_k$ for all $k\in\mathbb{Z}$.
Since
$$
f(x_k)=f(y_k)\in B[y_{k+1},\delta]=B[x_{k+1},\epsilon]=\mathfrak{F}_\epsilon(x_{k+1})
$$
for all $k\in\mathbb{Z}$,
we have that $(x_k)_{k\in\mathbb{Z}}$ is an $(\mathfrak{F},\epsilon)$-orbit.
Since
$d(x_k,y_k)=0\leq\epsilon$ for all $k\in\mathbb{Z}$, $(x_k)_{k\in\mathbb{Z}}$ $\epsilon$-shadows $(y_k)_{k\in\mathbb{Z}}$ proving the result.
\end{proof}

The definition of shadowing respect to a foliation extends the classical shadowing property \cite{ah}.
Recall that a homeomorphism $f:M\to M$ has the shadowing property if
for every $\epsilon>0$ there is $\delta>0$ such that every $\delta$-pseudo orbit
$(x_k)_{k\in\mathbb{Z}}$ can be $\epsilon$-shadowed i.e. there is $x\in M$ such that
$d(f^k(x),x_k)<\epsilon$ for all $k\in\mathbb{Z}$.

\begin{exa}
Every homeomorphism with the shadowing property has the shadowing property with respect to any foliation.
Conversely, every homeomorphism with the shadowing property with respect to any foliation has the shadowing property. This follows since the shadowing property with respect to the trivial foliation by points reduces to the shadowing property.
\end{exa}

We reduce the shadowing property above to the following one:

\begin{definition}
Let $f:M\to M$ be a homeomorphism preserving $\mathfrak{F}$.
We say that $f$ has the {\em finite shadowing property} with respect to $\mathfrak{F}$ if
for every $\epsilon>0$ there is $\delta>0$ such that
for every $\delta$-chain $x_0,\cdots, x_n$
there is an $(\mathfrak{F},\epsilon)$-chain $y_0,\cdots, y_n$ such that
$$
d(x_k,y_k)\leq\epsilon,\quad\quad 0\leq k\leq n.
$$
\end{definition}

We obtain the following result (compare with Lemma 8 in \cite{w}).

\begin{lem}
\label{lei}
Let $f:M\to M$ be a homeomorphism preserving a foliation $\mathfrak{F}$. 
If $f$ has the finite shadowing property with respect to $\mathfrak{F}$, then $f$ has the shadowing property with respect to $\mathfrak{F}$.
\end{lem}

\begin{proof}
Fix $\epsilon>0$ and let $\delta$ be given by the finite shadowing property.
Let $(x_k)_{k\in\mathbb{Z}}$ be a $\delta$-pseudo orbit of $f$.
Then, by the finite plaque shadowing, for all $m\in\mathbb{N}$ there is an $(\mathfrak{F},\epsilon)$-chain
$x_{-m}^m,\cdots, x^m_m$ such that
$$
d(x^m_k,x_k)\leq \epsilon,\quad\quad\forall -m\leq k\leq m.
$$
By standard Cantor diagonal argument there exists
a strictly increasing sequence $(m_j)_{j \in\mathbb{N}}$ with $m_j \to \infty$
and a bi-infinite sequence $(y_k)_{k \in \mathbb{Z}}$ in $X$,
such that for every fixed $k \in \mathbb{Z}$,
$$
x^{m_j}_k \ \longrightarrow\ y_k \quad \text{as } j \to \infty.
$$
Without loss of generality we can assume that $m_j=m$ for all $j\in\mathbb{N}$.

Fixing $k$ and letting $m\to\infty$ in the previous inequality we obtain
\begin{equation}
\label{canetada}
d(y_k,x_k)\leq \epsilon,\quad\quad\forall k\in\mathbb{Z}.
\end{equation}
On the other hand, since $x_{-m}^m,\cdots, x^m_m$ is an $(\mathfrak{F},\epsilon)$-chain we obtain
$$
f(x^m_k)\in\mathfrak{F}_\epsilon(x^m_{k+1}),
\quad\quad\forall m\in\mathbb{N}, -m\leq k\leq m-1.
$$
Then, fixing $k$ and letting $m\to\infty$ we get
$$
f(y_k)\in \mathfrak{F}_\epsilon(y_{k+1}),\quad\quad\forall k\in\mathbb{Z}.
$$
Hence $(y_k)_{k\in\mathbb{Z}}$ is a $(\mathfrak{F},\epsilon)$-orbit satisfying
\eqref{canetada}. Therefore, $f$ has the plaque shadowing property
and we are done.
\end{proof}

We say that $x\in M$ is {\em $\mathfrak{F}$-nonwandering} if for every saturated open neighborhood $U$ of $x$ there is $n\geq1$
such that $f^n(U)\cap U\neq\emptyset$.
We say that $x$ is {\em $\mathfrak{F}$-periodic}
if $\mathfrak{F}(f^k(x))=\mathfrak{F}(x)$ for some $k\geq1$.
Given $\delta>0$, an {\em $(\mathfrak{F},\delta)$-chain} from $x$ to $y$ is a finite sequence
$x_0,\cdots, x_r$ with $r\geq1$ such that $x_0=x$, $x_r=y$ and
$$
dist(f(x_i),\mathfrak{F}_\delta(x_{i+1}))\leq\delta,\quad\quad\forall 0\leq i\leq r-1.
$$

On the other hand, write $x\approx_{\mathfrak{F},\delta} y$ whenever
$(\mathfrak{F},\delta)$-chains from $x$ to $y$ and vice versa exist.
We write $x\approx_{\mathfrak{F}} y$ if $x\approx_{\mathfrak{F},\delta} y$ for every $\delta>0$.
We say that $x\in M$ is {\em $\mathfrak{F}$-chain recurrent} if
$x\approx_{\mathfrak{F}} x$.

Denote $Per^\mathfrak{F}(f)$, $\Omega^\mathfrak{F}(x)$ and $CR^\mathfrak{F}(f)$ the set of $\mathfrak{F}$-periodic, $\mathfrak{F}$-nonwandering and $\mathfrak{F}$-chain recurrent points respectively.
We have these sets are invariant, $\Omega^\mathfrak{F}(f)$ is closed nonempty and
$Per^\mathfrak{F}(f)\subset \Omega^\mathfrak{F}(f)$.
Furthermore,
$Per(f)\subset Per^\mathfrak{F}(f)$ and $\Omega(f)\subset \Omega^\mathfrak{F}(f)$.

With these definitions we state the following proposition.

\begin{pro}
\label{le1}
Let $f:M\to M$ be a homeomorphism preserving $\mathfrak{F}$.
If $f$ is expansive and has the shadowing property (both with respect to $\mathfrak{F}$), then
$CR^\mathfrak{F}(f)\subset\overline{Per^\mathfrak{F}(f)}\subset \Omega^\mathfrak{F}(f).$
\end{pro}

\begin{proof}
It suffices to prove
$CR^\mathfrak{F}(f)\subset \overline{Per^\mathfrak{F}(f)}.$
Fix $x\in CR^\mathfrak{F}(f)$ and arbitrary $\epsilon>0$.
Take $\epsilon_0=\epsilon$ and let
$e$ be an expansivity constant with respect to $\mathfrak{F}$ for this $\epsilon_0$.
Take $\delta>0$ from the shadowing property with respect to $\mathfrak{F}$ for $\epsilon'=\frac{1}8\min(e,\epsilon)$.

Since $x\in CR^\mathfrak{F}(f)$, there is $(\mathfrak{F},\delta)$-chain
$x_0,\cdots, x_r$ from $x$ to itself.
Define
$$
x_k=x_l\mbox{ whenever }k=nr+l\mbox{ for some }n\in\mathbb{Z}\mbox{ and }0\leq l\leq r-1.
$$
Then, $(x_k)_{k\in\mathbb{Z}}$ is a $\delta$-pseudo orbit and so
there is an $(\mathfrak{F},\epsilon')$-orbit $(y_k)_{k\in\mathbb{Z}}$ such that
$$
d(x_k,y_k)\leq\epsilon',\quad\quad\forall k\in\mathbb{Z}.
$$
Define
another sequence $(z_k)_{k\in \mathbb{Z}}$ by
$z_k=y_{k+r}$ for all $k\in\mathbb{Z}$.
Then, $(z_k)_{k\in\mathbb{Z}}$ is a $(\mathfrak{F},e)$-orbit
with
$$
d(y_k,z_k)\leq d(x_k,y_k)+d(x_k,y_{k-r})=d(x_k,y_k)+d(x_{k+r},y_{k+r})\leq 2\epsilon'<e,
$$
for all $k\in \mathbb{Z}$.

Then,
$$
y_0\in \mathfrak{F}_{\epsilon_0}(y_{r})
$$
by the expansivity with respect to $\mathfrak{F}$. In particular,
$$
\mathfrak{F}(y_0)=\mathfrak{F}(y_r).
$$
Now, it follows from the definition that
$$
y_{r-1}\in\mathfrak{F}(y_{r-1}).
$$
So,
$$
f(y_{r-1})\in \mathfrak{F}(f(y_{r-1})).
$$
But also
$$
f(y_{r-1})\in \mathfrak{F}_e(y_r)\subset \mathfrak{F}(y_r)
$$
so
$$
\mathfrak{F}(f(y_{r-1}))=\mathfrak{F}(y_r).
$$
Repeating the process we get
$$
\mathfrak{F}(f^r(y_0))=\mathfrak{F}(y_r).
$$
It follows that
$$
\mathfrak{F}(f^r(y_0))=\mathfrak{F}(y_0)
$$
hence
$$
y_0\in Per^\mathfrak{F}(f).
$$
Since
$$
d(x,y)=d(x_0,y_0)\leq \epsilon'<\epsilon
$$
and $\epsilon>0$ is arbitrary,
$x\in \overline{Per^\mathfrak{F}(f)}$ completing the proof.
\end{proof}

More examples are as follows.
Recall that a foliation $\mathfrak{F}$ is {\em uniformly compact} when all leaves are compact with finite holonomy.
In such a case the quotient topology of $M/\mathfrak{F}$ is compact and generated by the Hausdorff distance between the leaves of $\mathfrak{F}$ (c.f. \cite{e}).
We denote by $f/\mathfrak{F}:M/\mathfrak{F}\to M/\mathfrak{F}$ the quotient map.

\begin{exa}
Let $f:M\to M$ be homeomorphism of a closed manifold with an invariant uniformly compact foliation $\mathfrak{F}$.
Then,
\begin{itemize}
\item
If $f$ is expansive with respect to $\mathfrak{F}$, then $f/\mathfrak{F}$ is expansive.
\item
If $f$ has the shadowing property with respect to $\mathfrak{F}$, then
$f/\mathfrak{F}$ has the shadowing property.
\end{itemize}
\end{exa}

Since there are no expansive homeomorphisms of $S^1$, we can apply the above example in the one below.

\begin{exa}
No homeomorphism of the two torus $T^2$ can be
expansive with respect to the foliation by vertical circles of $T^2$.
\end{exa}

It was proved by \cite{hps} (see also \cite{psw}) that every diffeomorphism $f:M\to M$ which is normally hyperbolic and expansive with respect to a $C^1$ foliation $\mathfrak{F}$ satisfies the following kind of structural stability with respect to $\mathfrak{F}$: for every $C^1$ small perturbation $g:M\to M$ there are a foliation $\mathfrak{F}_g$ which is invariant under $g$ and a homeomorphism $h:M\to M$ sending the leaves of $\mathfrak{F}$ onto that of $\mathfrak{F}$ in an equivariant way.
The following result deals with homeomorphisms $f:M\to M$ preserving a foliation $\mathfrak{F}$ and with $C^0$ small perturbations of it.

\begin{pro}
\label{thA}
Let $f:M\to M$ be a homeomorphism preserving $\mathfrak{F}$.
If $f$ is expansive and has the shadowing property (both with respect to $\mathfrak{F}$), then $f$ is topologically stable with respect to $\mathfrak{F}$.
\end{pro}

\begin{proof}
We have that $f$ is uniformly expansive with respect to $\mathfrak{F}$ by Lemma \ref{l1}.
Take $\epsilon_0>0$.
Fix a uniform expansivity constant $e$ for this $\epsilon_0$.
Let $\delta$ be given by the shadowing property with respect to $\mathfrak{F}$
for
$$
\epsilon'=\frac{1}8\min\{e,\epsilon_0\}.
$$
Let $g:M\to M$ be a homeomorphism such that
$$
d_{C^0}(f,g)\leq\delta.
$$

We shall prove that
$$
H:M\to 2^M
$$
defined by
$$
H(x)=\{y\mid \exists (\mathfrak{F},\epsilon')\mbox{-orbit }
(y_k)_{k\in \mathbb{Z}}\mbox{ s.t. } y=y_0\mbox{ and }
d(y_k,g^k(x))\leq\epsilon',\,\forall k\in\mathbb{Z}\}
$$
is compact-valued, satisfies \eqref{alto} (with $\epsilon$ replaced by $\epsilon_0$) and is continuous with respect to $\mathcal{F}$ with valuation constant $\epsilon_0$ .

First we show
$$
H(x)\neq\emptyset,\quad
\quad\forall x\in M.
$$
Given $x\in M$ since $d_{C^0}(f,g)\leq\delta$, the sequence
$(g^k(x))_{k\in\mathbb{Z}}$ is a $\delta$-pseudo orbit of $f$.
Then, $(g^k(x))_{k\in\mathbb{Z}}$ can be $\epsilon'$-shadowed by an $(\mathfrak{F},\epsilon')$-orbit $(y_k)_{k\in\mathbb{Z}}$.
Since $y_0\in H(x)$ by definition, we obtain the assertion.

To prove that $H(x)$ is compact-valued we can use the Cantor diagonal argument to sequences of $(\mathfrak{F},\epsilon')$-orbits as before.
Next, by replacing $k=0$ in the definition of $H(x)$ we get
$$
d(x,y)\leq\epsilon'<\epsilon_0\quad\quad(\forall y\in H(x))
$$
proving the first expression in \eqref{alto}.

For the second expression take any $y\in H(x)$. Then, $y=y_0$ for some $(\mathfrak{F},\epsilon')$-orbit
$(y_k)_{k\in\mathbb{Z}}$ with
$$
d(y_k,g^k(x))\leq\epsilon',\quad\quad\forall k\in\mathbb{Z}.
$$
Define $\hat{y}_k=y_{k+1}$ for $k\in\mathbb{Z}$.
Then, $(\hat{y}_k)_{k\in\mathbb{Z}}$ is an $(\mathfrak{F},\epsilon')$-orbit
satisfying
$$
d(\hat{y}_k,g^k(g(x)))\leq\epsilon',\quad\quad\forall k\in\mathbb{Z}.
$$
Then,
$$
y_1=\hat{y}_0\in H(g(x)).
$$
But
$$
f(y)=f(y_0)\in \mathfrak{F}_{\epsilon'}(y_1)\subset\mathfrak{F}_{\epsilon_0}(y_1)
$$
so
$$
f(y)\in \bigcup_{y_1\in H(g(x))}\mathfrak{F}_{\epsilon_0}(y_1)
$$
proving the second expression in \eqref{alto}.

Afterwards, we prove that $H$ is $\mathfrak{F}$-valued with valuation constant $\epsilon_0$ namely
\begin{equation}
\label{marvin}
H(x)\subset \bigcap_{y\in H(x)}\mathfrak{F}_{\epsilon_0}(y),\quad\quad\forall x\in M.
\end{equation}
Take $x\in M$ and $y,y'\in H(x)$.
Then,
$$y=y_0\quad\mbox{ and }\quad
y'=y'_0
$$ for some
$(\mathfrak{F},\epsilon')$-orbits
$$
(y_k)_{k\in\mathbb{Z}}
\quad\mbox{ and }\quad
(y_k')_{k\in\mathbb{Z}}$$ with
$$
d(y_k,g^k(x))\leq\epsilon'\quad\mbox{ and }\quad
d(y_k',g^k(x))\leq\epsilon',\quad\quad\forall k\in\mathbb{Z}.
$$
Then,
$$
d(y_k,y_k')\leq d(y_k,g^k(x))+d(y_k',g^k(x))\leq 2\epsilon'<e,\quad\quad\forall k\in\mathbb{Z}.
$$
Since $e$ is a uniform expansivity constant for $\epsilon_0$,
$e$ is also an expansivity constant for $\epsilon_0$. Hence
$
y'=y'_0\in \mathfrak{F}_{\epsilon_0}(y_0)=\mathfrak{F}_{\epsilon_0}(y)
$
proving \eqref{marvin}.

To finish, we prove that $H$ is continuous with respect to $\mathcal{F}$ with valuation constant $\epsilon_0$.
Let $\rho>0$ and $N$ be given by the fact that $e$ is a uniform expansivity constant for $\epsilon_0$ (see Lemma \ref{l1}).
Since $g$ is continuous and $M$ compact, $g$ is uniformly continuous. So, there is $\Delta>0$ such that
$$
d(x,x')\leq\Delta\quad\Longrightarrow \quad d(g^k(x),g^k(x'))\leq \frac{e}3,\quad\forall -N\leq k\leq N.
$$
Take $x,x'\in M$ with
$$
d(x,x')\leq\Delta.
$$
Take also
$$
y'\in H(x')\quad\mbox{ and }\quad y\in H(x).
$$
Then, $y'=y'_0$ and $y=y_0$ for some
$(\mathfrak{F},\epsilon')$-orbits
$(y'_k)_{k\in\mathbb{Z}}$ and $(y_k)_{k\in\mathbb{Z}}$ satisfying
$$
d(y'_k,g^k(x'))\leq\epsilon'\quad\mbox{ and }\quad d(y_k,g^k(x))\leq\epsilon',\quad\forall k\in\mathbb{Z}.
$$
So, for all
$-N\leq k\leq N$ one has
\begin{eqnarray*}
d(y_k,y'_k) & \leq &  d(y'_k,g^k(x'))+d(g^k(x'),g^k(x))+d(y_k,g^k(x))\\
& \leq & 2\epsilon'+\frac{e}3\\
& < & \frac{7 e}{12}\\
& < & e.
\end{eqnarray*}
Thus, by uniform expansivity, since $y'_0=y'$ and $y_0=y$,
$$
dist(y,\mathfrak{F}_{\epsilon_0}(y'))\leq \rho.
$$
Since $y\in H(x)$ and $y'\in H(x')$ are arbitrary, $H$ is continuous with respect to $\mathfrak{F}$ with constant $\epsilon_0$.
This completes the proof.
\end{proof}

Let us give two short applications of the above proposition. The first one is as follows.

\begin{exa}
Every homeomorphism of a closed manifold is topologically stable with respect to the foliation with a unique leaf.
\end{exa}

\begin{proof}
This follows from Proposition \ref{thA} since every homeomorphism is expansive and has the shadowing property with respect to that foliation (by examples \ref{lotus} and \ref{interferiu} respectively).
\end{proof}

The second one is given below.

\begin{exa}
By taking the trivial foliation by points in Proposition \ref{thA} we
obtain that every expansive homeomorphism with the shadowing property of a closed manifold is topologically stable. This is precisely Walters's stability theorem \cite{w}.
\end{exa}

Next, we prove our last proposition.

\begin{pro}
\label{thB}
Let $f:M\to M$ be a homeomorphism preserving $\mathfrak{F}$.
If $f$ is topologically stable with respect to $\mathfrak{F}$, then
$f$ has the shadowing property with respect to $\mathfrak{F}$.
\end{pro}

\begin{proof}
If $dim(M)=1$, then $M=S^1$ is the circle and $\mathfrak{F}$ is either the trivial foliation by points or the foliation with only one leaf. In the first case, the topological stability with respect to $\mathfrak{F}$ coincides with the topological stability for the trivial foliation by points
(see Example \ref{air}). So, $f$ is a topologically stable circle homeomorphisms.
Then, $f$ has the shadowing property as it is topologically equivalent to a Morse-Smale diffeomorphisms \cite{y}.
In the second case, $f$ has the shadowing property with respect to $\mathfrak{F}$ (see Example \ref{interferiu}).
Therefore, we can assume that $dim(M)\geq2$.

By Lemma \ref{lei} it suffices to prove that $f$ has the finite shadowing property
with respect to $\mathfrak{F}$. For this we follow closely the proof of Theorem 11 in \cite{w}.

Fix $\epsilon>0$.
Let $\delta$ be given by the topological stability of $f$ with respect to $\mathfrak{F}$ for $\epsilon_0=\frac{\epsilon}2$. Let $x_0,\cdots, x_n$ be such that
$$
d(f(x_k),x_{k+1})\leq\frac{\delta}{4\pi},\quad\quad\forall 0\leq k\leq n-1.
$$
Since $dim(M)\geq2$, lemmas 9 and 10 in \cite{w} imply that there are a finite sequence
$x_0',\cdots, x_n'\in M$ and a homeomorphism
$h:M\to M$ such that
\begin{enumerate}
\item[(i)]
$d_{C^0}(h,i_M)\leq\delta$;
\item[(ii)]
$h(f(x_k'))=x_{k+1}'$ ($0\leq k\leq n-1$);
\item[(iii)]
$d(x_k,x_k')\leq \frac{\epsilon}2$ ($0\leq k\leq n$).
\end{enumerate}
Let
$$
g=h\circ f.
$$
Then, $g:M\to M$ is a homeomorphism which by (i) satisfies
$$
d_{C^0}(f,g)\leq\delta.
$$
So, by topological stability with respect to $\mathfrak{F}$,
there is a compact-valued map
$H:M\to 2^M$ which is continuous with respect to $\mathfrak{F}$ with constant $\epsilon_0$
satisfying
\begin{equation}
\label{constitucion}
d_{C^0}(H,i_M)\leq\epsilon_0=\frac{\epsilon}2\quad\mbox{ and }\quad
f(H(x))\subset \bigcup_{z\in H(g(x))}\mathfrak{F}_{\frac{\epsilon}2}(z)
\quad(\forall x\in M).
\end{equation}
On the other hand, the definition of $g$ and (ii) above imply
$$
g(x_k')=x'_{k+1},\quad\quad\forall 0\leq k\leq n-1.
$$
Now, pick
$$
y_0\in H(x_0').
$$
Then, \eqref{constitucion} provides
$$
y_1\in H(g(x_0'))=H(x_1')
$$
such that
$$
f(y_0)\in \mathfrak{F}_{\frac{\epsilon}2}(y_1).
$$
Again \eqref{constitucion} provides
$$
y_2\in H(g(x_1')=H(x'_2)
$$
such that
$$
f(y_1)\in \mathfrak{F}_{\frac{\epsilon}2}(y_2).
$$
Repeating the process we obtain
$y_0,\cdots, y_n\in X$ such that
$$
y_k\in H(x_k')\quad (\forall 0\leq k\leq n)\mbox{ and }
f(y_k)\in\mathfrak{F}_{\frac{\epsilon}2}(y_{k+1})
\quad(\forall 0\leq k\leq n-1).
$$
The second of the above expressions implies that $y_0,\cdots, y_n$ is an $(\mathfrak{F},{\frac{\epsilon}2})$-chain, and so, an
$(\mathfrak{F},{\epsilon})$-chain. Then, (iii) above and \eqref{constitucion} imply
$$
d(y_k,x_k)\leq d(y_k,x_k')+d(x_k',x_k)\leq
\frac{\epsilon}2+\frac{\epsilon}2=\epsilon, \quad\forall 0\leq k\leq n. 
$$
Therefore, $f$ has the finite shadowing property with respect to $\mathfrak{F}$ and we are done.
\end{proof}

Finally, we use the propositions to prove the theorems.

\begin{proof}[Proof of Theorem \ref{taylor}]
Let $f:M\to M$ be a dynamically coherent plaque expansive partially hyperbolic diffeomorphism of a closed manifold.
Then, $f$ is expansive with respect to the central foliation by definition. Moreover,
has the central shadowing property (in the sense of Definition 8 in \cite{kt}).
So, $f$ has the shadowing property with respect to the center foliation as well.
Therefore, $f$ is topologically stable with respect to the central foliation by Proposition \ref{le1}, and so, it is plaque topologically stable. This completes the proof.
\end{proof}

\begin{proof}[Proof of Theorem \ref{cirus}]
Let $f:M\to M$ be a partially hyperbolic diffeomorphism of a closed manifold.
Suppose that $f$ is expansive and topologically stable with respect to a center foliation $\mathfrak{F}^c$.
Then, $f$ also has the shadowing property with respect to $\mathfrak{F}^c$ by Proposition \ref{thB}.
Therefore, the result from Proposition \ref{thA} and the identities $Per^c(f)=Per^{\mathfrak{F}^c}(f)$, $\Omega^c(f)=\Omega^{\mathfrak{F}^c}(f)$ and $CR^c(f)=CR^{\mathfrak{F}^c}(f)$.
\end{proof}

\section*{Acknowledgements}

\noindent
The author gratefully acknowledges the anonymous referee for their valuable comments and suggestions, which helped to improve the clarity and presentation of this paper.

\section*{Funding}

\noindent
B. Shin was supported by the National Research Foundation of Korea(NRF) grant funded by the Korea government(MSIT) (No. RS-2021-NR061831).

\section*{Declaration of competing interest}

\noindent
There is no competing interest.

\section*{Data availability}

\noindent
No data was used for the research described in the article.

\end{document}